\documentclass{amsart}
\usepackage{amssymb,amsthm}
\usepackage{hyperref}
\usepackage{graphicx}
\newtheorem{theorem}{Theorem}[section]
\newtheorem{lemma}[theorem]{Lemma}

\theoremstyle{definition}
\newtheorem{definition}[theorem]{Definition}

\newcommand{\cM}{\mathcal M}

\newcommand{\V}{\mathrm V}
\newcommand{\A}{\mathrm A}

\newcommand{\ZZ}{\mathbb Z}

\newcommand{\soc}{\mathrm{soc}}
\newcommand{\Aut}{\mathrm{Aut}}
\newcommand{\Cay}{\mathrm{Cay}}
\newcommand{\Sym}{\mathrm{Sym}}

\newcommand{\cD}{\mathcal D}

\makeatletter
\def\imod#1{\allowbreak\mkern10mu({\operator@font mod}\,\,#1)}
\makeatother

\begin{document}

\title[Groups generated by two elements]{Groups of order at most $6\,000$ generated by two elements, one of which is an involution, and related structures}

\author[P. Poto\v{c}nik]{Primo\v{z} Poto\v{c}nik}
\address{Primo\v{z} Poto\v{c}nik, Faculty of Mathematics and Physics,
 University of Ljubljana, Jadranska 19, 1000 Ljubljana, Slovenia}\email{primoz.potocnik@fmf.uni-lj.si}

\author[P. Spiga]{Pablo Spiga}
\address{Pablo Spiga, University of Milano-Bicocca, Dipartimento di Matematica Pura e Applicata, Via Cozzi 55, 20126 Milano Italy} \email{pablo.spiga@unimib.it}
 
\author[G. Verret]{Gabriel Verret}
\address{Gabriel Verret, Centre for Mathematics of Symmetry and Computation, The University of Western Australia, 35 Stirling Hwy, Crawley, WA 6009, Australia, and\newline FAMNIT, University of Primorska, Glagolja\v{s}ka 8, SI-6000 Koper, Slovenia} 
\email{gabriel.verret@uwa.edu.au}

\thanks{The first author is supported by Slovenian Research Agency, projects P1--0294, J1-5433, L1--4292, and J1-6720. The third author is supported by UWA as part of the Australian Research Council grant DE130101001.}

\subjclass[2010]{Primary 05E18; Secondary 20B25}
\keywords{$2$-generated groups, regular maps, chiral maps, Cayley graphs, arc-transitive digraphs}

\begin{abstract}
A \emph{$(2,*)$-group} is a group that can be generated by two elements, one of which is an involution. We describe the method we have used to produce a census of all $(2,*)$-groups of order at most $6\,000$. Various well-known combinatorial structures are closely related to $(2,*)$-groups and we also obtain censuses of these as a corollary.
\end{abstract}

\maketitle

\section{Introduction}
The objects that play a central role in our paper are \emph{$(2,*)$-groups}, that is, groups that can be generated by two (not necessarily distinct) elements, one of which is an involution. We will also need the notion of a $(2,*)$-triple, which we now define.

\begin{definition}
A {\em $(2,*)$-triple} is a triple $(G,x,g)$ such that $G$ is a $(2,*)$-group, $\{x,g\}$ is a generating set for $G$ and $x$ is an involution.
Two $(2,*)$-triples $(G_1,x_1,g_1)$ and $(G_2,x_2,g_2)$ are \emph{isomorphic} if there exists a group isomorphism from $G_1$ to $G_2$ mapping $x_1$ to $x_2$ and $g_1$ to $g_2$.	
\end{definition}

The first aim of this paper in to announce a complete determination of all $(2,*)$-groups of order at most $6\,000$. The methods we used and how they improve on the ones used by previous authors are discussed in Section~\ref{sec:comp}. Here, we only state the following:

\begin{theorem}
\label{thm:main2*}
Up to isomorphism, there are precisely $129\,340$ $(2,*)$-groups and $345\, 070$ $(2,*)$-triples of order at most $6\,000$.
\end{theorem}

The database of all $(2,*)$-groups and triples in a form readable by {\sc magma} \cite{magma} is available at \cite{PotWeb}.

\begin{figure}[hhh]
\begin{center}
\includegraphics[scale=0.45]{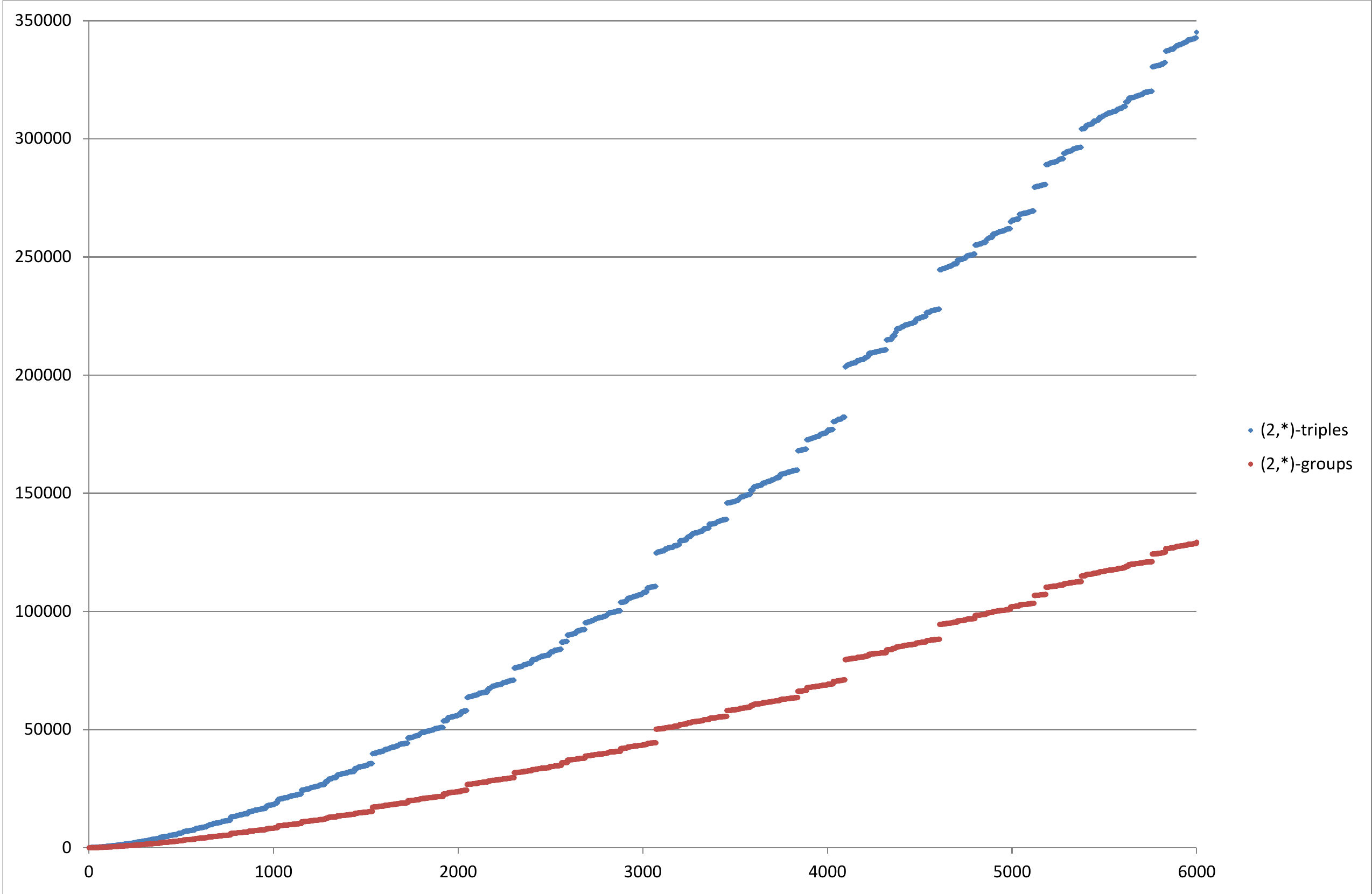}
\caption{\small{Number of $(2,*)$-groups and triples up to a given order.}}
\label{table:2*}
\end{center}
\end{figure}

The second aim of this paper is to prove an asymptotic enumeration result for $(2,*)$-groups and $(2,*)$-triples. Let $f(n)$ and $f_t(n)$ denote the number (up to isomorphism) of $(2,*)$-groups and the number of $(2,*)$-triples (respectively) of order at most $n$. The graphs of $f(n)$ and $f_t(n)$ are depicted in Figure~\ref{table:2*}. A quick look at this picture might suggest that both $f(n)$ and $f_t(n)$ grow polynomially in $n$. This is not the case. In fact, in Section~\ref{sec:5}, we show the following:

\begin{theorem}\label{thrm:asy}
There exist positive constants $a$ and $b$ such that, for $n\geq 2$ we have
$$n^{a\log n}\leq f(n)\leq f_t(n)\leq n^{b\log n}.$$
\end{theorem}

The problem of optimising the constants $a$ and $b$ in Theorem~\ref{thrm:asy} is beyond the scope of this article and is related to the problem of enumerating the  normal subgroups of finite index in certain finitely presented groups (see for example~\cite[Chapter~$2$]{LS}).

The third aim of the paper is a discussion of a relationship between  $(2,*)$-groups and 
 several highly symmetrical geometric and combinatorial objects; for example, cubic Cayley graphs, arc-transitive digraphs of out-valence $2$ and rotary maps (both chiral and reflexible, on orientable and non-orientable surfaces). These relationships are explained in Sections~\ref{sec:relationship} and~\ref{sec:maps}. Together with our census of $(2,*)$-triples they have allowed us to generate complete lists of:
\begin{itemize}
\setlength{\itemsep}{0pt}
\item cubic Cayley graphs generated by an involution and a non-involution, with at most $6\,000$ vertices;
\item digraphs of out-valence $2$ admitting an arc-regular group of automorphisms, with at most $3\,000$ vertices;
\item rotary maps (both chiral and reflexive) on orientable surfaces, with at most $3\,000$ edges;
\item regular maps on non-orientable surfaces, with at most $1\, 500$ edges.
\end{itemize}

Databases of these objects are also available at \cite{PotWeb}.

\section{Constructing the census of small $(2,*)$-groups}
\label{sec:comp}

In this section, we are concerned with the problem of generating a complete list of $(2,*)$-triples $(G,x,g)$ with $|G| \le m$ for some prescribed constant $m$.
Let us discuss a few possible approaches to this problem.

\subsection{Using a database of small groups}\label{sec:database}

If $m$ is sufficiently small, then  a database of all the groups of order at most $m$ might be available. For example, at the time of writing of
this article, all groups of order $2\,000$ are known, and all except those of order $1\, 024$ are available in standard distributions of {\sc GAP}~\cite{GAP} and {\sc Magma} \cite{magma}. One might thus try to search through such a database and, for each group $G$ in the database, determine all possible
generating pairs $(x,g)$ with $x$ being an involution, up to conjugacy in $\Aut(G)$.

While this approach is rather straight-forward, it has an obvious downside in that it requires iterating over all the groups of order at most $m$. Even getting access to the groups of order $1\,024$ is difficult at the moment and, even in the near future, the groups of order $2\,048$ will probably remain out of reach. Even if one had access to these groups, their number would make it all but impossible to iterate over them. (There are more than $10^{15}$ groups of order $2\, 048$~\cite{GNU}.)

These considerations should make it clear that, to make any significant progress, one should find a way to avoid having to consider all groups of order at most $m$.

\subsection{Using the {\tt LowIndexNormalSubgroups} algorithm}
\label{sec:lins}

Observe that every $(2,*)$-group is an epimorphic image of the free product $U:=C_2*C_\infty = \langle x,g \mid x^2\rangle$ and can thus be obtained
as a quotient of $U$ by a normal subgroup $N$ not containing $x$. Note that this yields not only the $(2,*)$-group $U/N$,
but also the $(2,*)$-triple $(U/N,Nx,Ng)$.  In order to find all $(2,*)$-triples of order at most $m$
it thus suffices to find all normal subgroups of $U$ of index at most $m$.

Firth and Holt~\cite{Firth} have developed a very efficient algorithm for determining normal subgroups of bounded index in a finitely presented group.
The current implementation of this algorithm in {\sc Magma} can, in principle, compute all normal subgroups of index at most $500\,000$. However,
for certain finitely presented groups the practical limitations of the algorithm 
(or at least its current implementation in {\sc Magma}) make the computation unfeasible, even for much smaller indices. 

%Since every quotient of $U$ of order at most $m$ is isomorphic to a quotient of $U_n := C_2*C_n$ for some $n$ satisfying $m/2 \le n \le m$,
%this algorithm can be run in parallel by computing quotients of each $U_n$, $m/2\le n\le m$, rather than of $U$. This trick does not only
%provide a means for parallelisation, but can in fact sometimes cut down on the total computation time. (IS THIS TOO MUCH DETAIL?)

An approach along these lines (disguised in the language of rotary maps; see Section~\ref{sec:ORmaps}) has been successfully used by Conder~\cite{conderPage} to determine all normal subgroups of $U$ of index at most $2\,000$,
but computations took several months.

\subsection{Using group extensions}\label{sex:extensions}

Finally, we describe the approach that we used to compile a complete list of $(2,*)$-groups and triples of order at most $6\,000$. The method is inductive and constructs $(2,*)$-groups as extensions of smaller ones. The general idea is not new (see for example~\cite{handbook}), but our recent implementation proved to be more efficient than  recent efforts using the {\tt LowIndexNormalSubgroups} algorithm.

 Let us first set some terminology. If $N$ is a normal subgroup of a group $G$ and $Q$ is a group isomorphic to the quotient $G/N$, then we  say that $G$ is an {\em extension of $Q$ by $N$}. 
(Some authors call $G$ an extension of $N$ by $Q$.) If $N$ is a minimal normal subgroup of $G$, then we shall say that the extension is {\em direct},
and if $N$ is elementary abelian, then we say that the extension is {\em elementary abelian}. The {\em soluble radical} of a group is its (unique) largest normal soluble subgroup. 

\begin{lemma}
\label{lem:rec}
If $G$ is a $(2,*)$-group, then either $G$ has a trivial soluble radical, or $G$ is a direct elementary abelian extension of a smaller $(2,*)$-group or of a cyclic group of odd order.
\end{lemma}

\begin{proof}
As $G$ is a $(2,*)$-group, we have $G=\langle x,g\rangle$, for some involution $x\in G$ and some $g\in G$. Suppose that the soluble radical $S$ of $G$ is non-trivial. Let $N$ be a minimal normal subgroup of $G$ contained in $S$.
Since $S$ is soluble, $N$ is elementary abelian and hence $G$ is a direct elementary abelian extension of $N$ by $G/N$. If $G/N=\langle xN,gN\rangle$ is not a $(2,*)$-group, then $xN=N$ and $gN$ has odd order, that is, $G/N=\langle gN\rangle$ is cyclic of odd order.
\end{proof}

Lemma~\ref{lem:rec} suggests an inductive procedure to construct $(2,*)$-groups from smaller ones. The base case of this inductive process are $(2,*)$-groups with trivial soluble radical and cyclic groups of odd order. If $G$ is a finite group with trivial soluble radical, then $\soc(G)$ (that is, the group generated by the minimal normal subgroups of $G$) is isomorphic to a direct product of non-abelian simple groups and, moreover, $G$ acts faithfully on $\soc(G)$ by conjugation and thus $G$ embeds into $\Aut(\soc(G))$. This allows one to use a database of small simple groups (available, say, in {\sc Magma} or {\sc GAP}) to construct all groups of order at most $m$ with  trivial soluble radical. 

For example, it is an easy computation to determine that there are precisely $23$ groups with trivial soluble radical of order at most $6\,000$. For a given group $G$ with trivial soluble radical, one can find all $(2,*)$-triples $(G,x,g)$ by determining all epimorphisms from the group $C_2*C_\infty$ to $G$ (where two epimorphisms are considered equivalent if they differ by some automorphism of $G$).

Let us now discuss the inductive step. Suppose we are given a group $Q$ of order $n$ (which, for our purposes, can be taken to be either a $(2,*)$-group or cyclic of odd order)
 and would like to find all direct elementary abelian extensions of $Q$ of order at most $m$.
In view of the general theory of group extensions, it suffices to find all irreducible $\ZZ_pQ$-modules $N$, with $N$ isomorphic to an elementary 
abelian group $\ZZ_p^d$, such that $p^d n\le m$
and then, for each such module $N$, compute the cohomology group $H^2(Q,N)$.
Each element of $H^2(Q,N)$ then gives rise to a direct extension of $Q$ by $N$, and 
conversely, each direct elementary abelian extension of $Q$ of order at most $m$ can be obtained in this manner.
Efficient algorithms for computing the irreducible modules of a given group and the corresponding second cohomology group are known (see for example~\cite{handbook})
and are implemented in {\sc Magma}. 

It is not surprising that computationally the hardest case is the extension of $2$-groups by $2$-groups. Fortunately, in this case some parts of the 
inductive step can be simplified. Namely, when $Q$ is a $2$-group and $p=2$, the only irreducible $\ZZ_2Q$-module is the  (trivial) $1$-dimensional $\ZZ_2Q$-module $\ZZ_2$ and hence only the cohomology group $H^2(Q,\ZZ_2)$ needs to be considered. This shortcut speeds up the determination of
$(2,*)$- $2$-groups considerably.

Once the direct elementary abelian extensions $G$ of $Q$ are determined, one needs to check which of them are $(2,*)$-groups and, for those which
are, find all pairs $(x,g)$ such that $(G,x,g)$ is a $(2,*)$-triple.
This can be done by first computing the automorphism group $\Aut(G)$, then choosing a representative of each orbit of $\Aut(G)$ on the set of involutions of $G$ then, for each representative $x$,
computing the stabiliser $\Aut(G)_x$ of $x$ in $\Aut(G)$, choosing a representative $g$ from each orbit of $\Aut(G)_x$ on $G$ and, finally, discarding the pairs $(x,g)$ that do not generate $G$.

As mentioned at the beginning of the section, this method is the one that we used in order to obtain the complete list of $(2,*)$-groups and triples of order at most $6\, 000$. The computation took a few weeks on a computer with a 2.93 GHz Intel Xeon processor and 56GB of memory.

\section{Proof of Theorem~\ref{thrm:asy}}
\label{sec:5}

Since every $(2,*)$-group gives rise to a $(2,*)$-triple, we have $f(n)\leq f_t(n)$. On the other hand, if $G$ is a $(2,*)$-group of order $n$, then there are at most $n^2$ choices for $(x,g)\in G\times G$ hence at most $n^2$ $(2,*)$-triples with first coordinate $G$. This shows that $f_t(n)\leq n^2f(n)$. In particular, it suffices to prove that there exist positive constants $a$ and  $b$ such that, for $n\geq 2$, we have $$n^{a\log n}\leq f(n)\leq n^{b\log n}.$$

Clearly, $f(n)$  is at most the number of groups  (up to isomorphism) of order at most $n$ generated by $2$ elements. By a celebrated theorem of Lubotzky~\cite[Theorem~$1$]{Lub}, the latter is at most $n^{b\log n}$. (Some information on the constant $b$ can be found in~\cite[Section~$3$, Remark~$1$]{Lub}.)

The lower bound follows easily from a theorem of M\"{u}ller and Schlage-Puchta: let $A$ be a cyclic group of order $2$, let $B$ be a cyclic group of order $3$ and let $G$ be the free product of $A$ and $B$, that is, $G=A\ast B$. Let $C=A\times B$, let $\pi:G\to C$ be the natural projection and let $N$ be the kernel of $\pi$.

Observe that, since $\pi$ is surjective, $N\cap A=N\cap B=1$. By Bass-Serre theory, $N$ is a free group (see~\cite[Theorem~$4$, page $27$]{Serre}). Observe also that $G$ has a natural action as a transitive group of automorphisms of the infinite $3$-valent tree $\mathcal{T}$. As $N\unlhd G$ and $|G:N|=|C|=6$, we see that $N$ has at most $6$ orbits on the vertices of $\mathcal{T}$. Assume that $N$ is cyclic and let $\alpha$ be a generator of $N$. From~\cite[Proposition~$3.2$(iii)]{Tits}, the element $\alpha$ acts as a translation on some infinite path of $\mathcal{T}$. As $\mathcal{T}$ has valency $3$, from this it follows immediately that $N$ has infinitely many orbits on the vertices $\mathcal{T}$, a contradiction. Therefore $N$ is non-cyclic and hence is a free group of rank at least $2$.

For each $n\in\mathbb{N}$, define $$\mathcal{N}_n=\{M\mid M\unlhd G, M\leq N, |G:M|\leq n\}.$$
As $N$ is a free group of rank at least $2$,~\cite[Theorem~$1$]{MS} yields that there exists a positive constant $a'$ with $|\mathcal{N}_n|\geq n^{a'\log n}$ for $n\geq 2$. Observe that, for every group $M\in\mathcal{N}_n$, the quotient $G/M$ is a $(2,*)$-group of order at most $n$. 

Fix $M\in\mathcal{N}_n$, the number of $M'\in\mathcal{N}_n$  with $G/M'\cong G/M$ is exactly the number of surjective homomorphisms from $G$ to $G/M$. Since $G$ is $2$-generated and $|G/M|\leq n$, the number of such homomorphisms is at most $|G/M|^2\leq n^2$. We conclude that $f(n)\geq |\mathcal{N}_n|/n^2\geq (n^{a'\log n})/n^2$ and the result follows.

\section{$(2,*)$-groups and graphs}
\label{sec:relationship}

\subsection{Cubic Cayley graphs.}
\label{sec:cay}

Let $G$ be a group and let $S$ be a generating set for $G$ which is inverse-closed and does not contain the identity. The \emph{Cayley graph on $G$ with connection set $S$} is the graph with vertex-set $G$ and two vertices $u$ and $v$ adjacent if $uv^{-1}\in S$. It is denoted $\Cay(G,S)$. It is easy to see that $\Cay(G,S)$ is a connected vertex-transitive graph of valency $|S|$. 

Cayley graphs form one of the most important families of vertex-transitive graphs. In fact, at least for graphs of small order, the overwhelming majority of vertex-transitive graphs are Cayley graphs. This makes them crucial in any project of enumeration of vertex-transitive graphs. 

With respect to valency, the first non-trivial case is the case of cubic graphs. Let $\Gamma=\Cay(G,S)$ be a cubic Cayley graph. Note that $S$ is an inverse-closed set of size three and thus must consist either of three involutions or have the form $\{x,g,g^{-1}\}$ where $x$ is an involution and $g$ is not. In the latter case, we say that $\Gamma$ has \emph{type I}. In this case, $(G,x,g)$ is a $(2,*)$-triple and thus type I graphs arise from $(2,*)$-triples.

While constructing the graphs of type I from the catalogue of $(2,*)$-triples is computationally easy, reduction modulo graph isomorphism requires a careful choice of computational tools. For example, {\sc magma} failed to finish the computation in reasonable time but the {\sc Sage} package~\cite{sage} performed considerably better and yielded the result in a few hours.
We would like to thank Jernej Azarija for his help in this matter, which allowed us to conclude that:

\begin{theorem}\label{thm:cubicCayley}
There are precisely $274\,171$ connected cubic Cayley graphs of type I with at most $6\,000$ vertices.
\end{theorem}

Moreover, by Theorem~\ref{thrm:asy}, there are at most $n^{b\log n}$  type I graphs of order at most $n$. On the other hand, non-isomorphic $(2,*)$-triples may give rise to isomorphic Cayley graphs. In general, it is very hard to control when two non-isomorphic $(2,*)$-triples give rise to isomorphic Cayley graphs and thus the lower bound in Theorem~\ref{thrm:asy} does not immediately give a lower bound on the number of graphs of type I. (See for example~\cite{enumeration} for more details on such lower bounds.)

Recently, we published a census of all cubic vertex-transitive graphs of order at most $1\,280$~\cite{cubiccensus}. The method we used to construct the graphs of type I was a mix of the ones described in Section~\ref{sec:database} and Section~\ref{sex:extensions} (see ~\cite[Section~$3$]{cubiccensus}) and would have been difficult to extend to orders greater than $2\,000$. The methods described in the current paper thus constitute an improvement, as they allowed us to reach order $6\,000$.

\subsection{Arc-transitive digraphs of out-valency two}
\label{sec:di}

A \emph{digraph} is an ordered pair $(V,A)$ where $V$ is a finite non-empty set and $A\subseteq V \times V$ is a binary relation on $V$. If $\Gamma=(V,A)$ is a digraph, then we shall refer to the set $V$ and the relation $A$ as the {\em vertex-set} and the {\em arc-set} of $\Gamma$, and denote them by $\V(\Gamma)$ and $\A(\Gamma)$, respectively. Members of $V$ and $A$ are called {\em vertices} and {\em arcs}, respectively. For a vertex $v$ of $\Gamma$, the number $|\{w\in \V(\Gamma)\mid (v,w)\in A(\Gamma)\}|$ is called the \emph{out-valency} of $v$. 

An \emph{automorphism} of a digraph $\Gamma$ is a permutation of $\V(\Gamma)$ which preserves the arc-set $\A(\Gamma)$. Let $G$ be a subgroup of the automorphism group $\Aut(\Gamma)$ of $\Gamma$. We say that $\Gamma$ is \emph{$G$-arc-transitive} provided that $G$ acts transitively on $\A(\Gamma)$. In this case, if $\Gamma$ is connected, then each of its vertices has the same out-valency, say $d$, and we say that $\Gamma$ has \emph{out-valency} $d$.

If $\Gamma$ is an arc-transitive digraph, then its arc-set $\A(\Gamma)$ is either {\em symmetric} (that is, for every arc $(u,v)\in \A(\Gamma)$,
also $(v,u)\in \A(\Gamma$)), or {\em asymmetric} (that is, for every $(u,v)\in \A(\Gamma)$, we have $(v,u)\notin\A(\Gamma)$).
We will think of a digraph with a symmetric arc-set as a {\em graph}.

Let $\Gamma$ be a connected $G$-arc-transitive digraph of out-valency two. It is easily seen that, for a vertex $v$ of $\Gamma$, the vertex-stabiliser $G_v$ has order $2^s$ for some $s\geq 1$. Moreover, $s=1$ if and only if $G$ acts regularly on $\A(\Gamma)$. In this case, let $x$ be the involution generating $G_v$ and let $g$ be an element of $G$ mapping $(u,v)$ to $(v,w)$, where $(u,v)$ and $(v,w)$ are arcs of $\Gamma$. It is not hard to show that $\{x,g\}$ generates $G$ and thus $(G,x,g)$ is a $(2,*)$-triple. Note also that $\langle x\rangle$ is not central in $G$ (as it is the point-stabiliser of a transitive permutation group). 
%*** ADDED BY PRIMOZ: Moreover, if $\Gamma$ is {\em asymmetric} (that is, if its arc-set is asymmetric), then $g^x\not=g^{-1}$. ***
Every digraph of out-valency $2$ with an arc-regular group of automorphisms thus arises from a $(2,*)$-triple with $x$ not central.

Conversely, given a $(2,*)$-triple $(G,x,g)$ such that $\langle x\rangle$ is not central in $G$, one can recover a $G$-arc-regular digraph of out-valency $2$ by the well-known coset graph construction: the vertices are the right cosets of $H=\langle x\rangle$ in $G$ with $(Ha,Hb)$ being an arc whenever $ba^{-1} \in HgH$. 

As in the previous section, checking for digraph isomorphism requires some computational work which was performed by Katja Ber\v{c}i\v{c} as a part of her doctoral thesis \cite{katja}. This allowed us to obtain:
%*** ADDED BY PRIMOZ:  Unless $g^x=g^{-1}$, the resulting digraph will be asymmetric ***.

\begin{theorem}\label{thm:ARD}
There are precisely $165\,952$ asymmetric connected digraphs of out-valency $2$ on at most $3\,000$ vertices, with an arc-regular group of automorphisms.
\end{theorem}

This census of digraphs was used in our recent census of all arc-transitive digraphs of out-valency two with at most $1\,000$ vertices~\cite{4HATcensus}.

As in Section~\ref{sec:cay}, Theorem~\ref{thrm:asy} implies that, up to isomorphism, there are at most $n^{b\log n}$ digraphs of out-valency $2$ with an arc-regular group of automorphisms but, again,  non-isomorphic $(2,*)$-triples may give rise to isomorphic digraphs and thus lower bounds are harder to obtain.

Finally, we note that the underlying graph of an asymmetric $G$-arc-transitive digraph $\Gamma$ of out-valency $d$ is a $2d$-valent graph 
on which $G$ acts \emph{half-arc-transitively} (that is, vertex- and edge- but not arc-transitively). 
Moreover, this process can be reversed (see for example~\cite[Section 2.2]{4HATcensus}) and we thus obtain the following:

\begin{theorem}\label{thm:4HAT}
There are precisely $76\, 200$ connected $4$-valent graphs on at most $3\,000$ vertices that admit a half-arc-transitive group of automorphisms with vertex-stabiliser of order $2$.
% of which $??$ are arc-transitive while $????$ are not.
\end{theorem}

\section{$(2,*)$-groups and maps}
\label{sec:maps} 
 
Intuitively, a map is a drawing of a graph onto a surface or, slightly more formally, 
it is an embedding of a graph onto a closed surface (either orientable or non-orientable) which decomposes the surface into open, simply connected regions, called {\em faces}. Each face can be decomposed further into {\em flags}, that is, triangles with one vertex in the centre of the face, one vertex in the centre of an edge and one in a vertex of the embedded graph. An automorphism of a map is then defined as a permutation of the flags induced by a homeomorphism of the surface that preserves the embedded graph.

It is well known that this geometric notion can also be viewed algebraically. In this paper, we adopt this algebraic point of view
and use the geometric interpretation only as a source of motivation. For a more thorough discussion on different aspects of maps, and the relationship between
their geometric and algebraic description, we refer the reader to \cite{JonSin,JonSin2}, or to the excellent survey~\cite{siran}.

Enumeration of maps, especially those exhibiting many symmetries, has a long history, going back to the Ancient Greeks and the classification of the Platonic solids.
In this section, we shall be interested in the enumeration and construction of all rotary maps (both reflexible and chiral, orientable and non-orientable) with
a small number of edges. Such an enumeration was first attempted by Wilson in~\cite{WilPhD} for the case of oriented rotary maps on at most $100$ edges.
More recently, a complete list of all rotary maps on at most $1\,000$ edges was obtained by Conder \cite{conderPage}. 

This section has no ambition to be a survey on maps and their symmetries; its main purpose is to show
how the database of $(2,*)$-groups was used to extend Conder's database \cite{conderPage} up to $3\,000$ edges in the orientable case and up to 
$1\,500$ edges in the non-orientable case.

\subsection{Monodromy groups of maps}

A faithful transitive action of a $(2,*)$-group on a set $\cD$ can be interpreted as
the {\em monodromy group} of a map on an orientable surface.
More precisely, if $(G,x,g)$ is a $(2,*)$-triple acting faithfully and transitively on a finite set $\cD$ in such a way that $x$ has no fixed points,
then one can construct a map with faces, edges and vertices corresponding to the orbits of the groups
$\langle g \rangle$, $\langle x \rangle$ and $\langle xg \rangle$, respectively, and with
incidence between these objects given in terms of non-empty intersection.
Conversely, every  map on a closed orientable surface can be obtained in this way from a $(2,*)$-triple. By considering transitive faithful actions of $(2,*)$-groups, one can thus obtain all graph embeddings into orientable surfaces.

\subsection{Oriented rotary maps}
\label{sec:ORmaps}

An {\em automorphism} of  the map $\cM$ associated with a $(2,*)$-triple $(G,x,g)$ acting on $\cD$ is any permutation of $\cD$ that commutes with $x$ and $g$,
and thus the automorphism group $\Aut(\cM)$ equals the centraliser of $G$ in $\Sym(\cD)$.

A very special case occurs when $\Aut(\cM)$ is transitive on the dart-set $\cD$, which occurs if and only if $G$ (and thus also $\Aut(\cM)$)
acts regularly on $\cD$. 
In that case one can identify $\cD$ with the elements of $G$ in such a way that $x$ and $g$ act upon $\cD=G$
as permutations $a\mapsto xa$ and $a\mapsto ga$ for all $a\in G$, respectively. The centraliser $\Aut(\cM)$ of $G$ in $\Sym(\cD)$
is then generated by the permutation $a \mapsto ax$ and $a \mapsto ag$. In this sense we may view the group $G$ 
as the automorphism group $\Aut(\cM)$ (rather than the monodromy group) acting regularly with right multiplication on the set of darts $\cD=G$.
In this setting, the elements $R=g$ and $S=g^{-1}x$ act as  one step-rotations around the centre of a face and around a vertex incident to that face,
respectively. We shall always assume that the underlying surface of the map is oriented in such a way that $R$ and $S$ rotate one step in the clock-wise sense;
note that the same map but with the opposite orientation is obtained from the triple $(G,g^{-1},gxg^{-1})$, giving rise to the rotations $R^{-1}$ and $S^{-1}$.
This justifies the following terminology.

\begin{definition}
An {\em oriented rotary map} is a triple $(G,R,S)$ such that $G$ is a group, $\{R,S\}$ is a generating set for  $G$ and $RS$ is an involution. Two oriented rotary maps $(G_1,R_1,S_1)$ and $(G_2,R_2,S_2)$ are \emph{isomorphic} if there exists a group isomorphism from $G_1$ to $G_2$ mapping $R_1$ to $R_2$ and $S_1$ to $S_2$.
\end{definition}

Given an oriented rotary map $(G,R,S)$ one can reverse the process and construct the associated $(2,*)$-triple $(G,RS,R)$. Moreover,
two oriented rotary maps are isomorphic if and only if the associated $(2,*)$-triples are isomorphic. Thus, there is a bijective correspondence between the isomorphism classes of $(2,*)$-triples and the isomorphism classes of oriented rotary maps.

Let us now define a few invariants and operations on oriented rotary maps that are motivated by their geometric interpretations as embeddings of graphs on surfaces.
Let $(G,R,S)$ be an oriented rotary map. A right coset of $\langle R\rangle$ in $G$ is called a {\em face}, a coset of $\langle S \rangle$ a {\em vertex},
and a coset of $\langle RS\rangle$ an {\em edge} of the map. 
The orders of $R$ and $S$ are thus called the {\em face-length} and the {\em valence} of the map, respectively, while the symbol $\{|R|,|S|\}$ is called the {\em type} of the map. Furthermore, since $|\langle RS \rangle|=2$, it follows that a the oriented rotary map $(G,R,S)$ has $|G|/2$ edges.
The {\em mirror image} of $(G,R,S)$ is the oriented rotary map $(G,R^{-1},S^{-1})$. If an oriented rotary map is isomorphic to its mirror image, it is called {\em reflexible} and is {\em chiral} otherwise. Our enumeration of $(2,*)$-triples (see Theorem~\ref{thm:main2*}) yields the following result. 

\begin{theorem}
\label{thm:mainMaps}
There are precisely $345\,070$ %(2 are on one edge - K2 or loop on sphere)
oriented rotary maps with at most $3\,000$ edges,
of which $122\,092$ are chiral and $222\,978$ are reflexible.
\end{theorem}

The number of reflexible and chiral oriented rotary maps with up to a given number of edges is depicted in Figure~\ref{table:chiral}.

\begin{figure}[h]
\begin{center}
\includegraphics[scale=0.45]{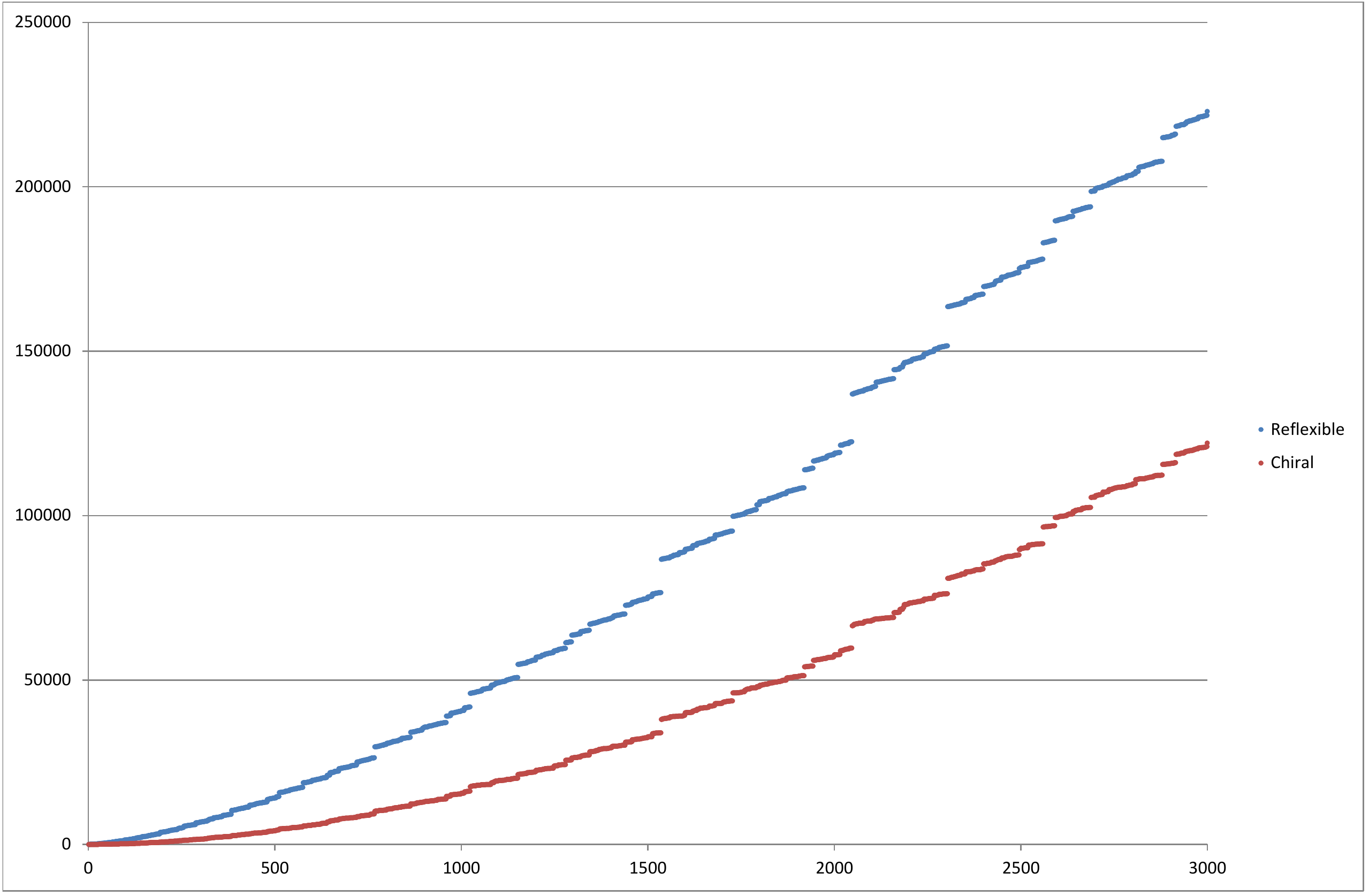}
\caption{\small{Number of chiral and reflexible oriented rotary maps with up to a given number of edges.}}
\label{table:chiral}
\end{center}
\end{figure}

\subsection{Regular maps}
\label{sec:regmaps}

Let $\cM=(G,R,S)$ be a reflexible oriented rotary map. By definition, there exists an automorphism $\tau$ of $G$, called the \emph{reflector} of $\cM$,
 with  $\tau(R)=R^{-1}$ and $\tau(S)=S^{-1}$. (Since $R$ and $S$ generate $G$, the reflector is unique and of order at most $2$).
 Let $C_2$ be a group of order $2$, let $b$ be its generator, let $\vartheta \colon C_2 \to \Aut(G)$ be the
homomorphism mapping $b$ to $\tau$, and let
$$A=G\rtimes_\vartheta C_2.$$
Further, let
$a=Rb\> \hbox{ and } c=b S$, and observe that
$a$ and $c$ are involutions such that $a\not = c$. 
Moreover, $ac=Rb\cdot bS=RS$ and $ca=bS\cdot Rb=S^{-1}R^{-1}=RS$ because $RS$ is an involution; in particular, $\langle a, c\rangle$ is the Klein $4$-group. 
Note also that $R=ab$ and $S=bc$, and therefore $\langle ab,bc\rangle$ has index $2$ in $A$.

Geometrically, the group $A$ can be viewed as the automorphism group of the 
orientable (but unoriented) map arising from $(G,R,S)$, with $\langle R, S\rangle$ corresponding to the group of orientation preserving automorphisms
and $b$ acting as the orientation reversing automorphism which reflects about the axis through the vertex corresponding to $\langle S \rangle$ and
the centre of the face corresponding to $\langle R \rangle$. In this setting, the automorphism $c$ can be viewed as the reflection over 
the edge $\{v, v^{R^{-1}}\}$, where $v$ is the vertex corresponding to $\langle S \rangle$, while $a$ reflects over the line perpendicular to that edge.
The group $A$ then acts regularly on the set of flags of the oriented rotary map.
%a triangle inside a face with one vertex being
%in the centre of the face, the other in a vertex on the boundary of the face and the third in the middle of an edge incident to that vertex and lying on the boundary 
%of the face).
This motivates the following definition:

\begin{definition}
\label{def:reg}
A {\em regular map} is a quadruple $(A,a,b,c)$ such that $A$ is a group, $a,b,c$ are involutions generating $A$ and $|\langle a,c\rangle| =4$.
Two regular maps
$(A_1,a_1,b_1,c_1)$ and $(A_2,a_2,b_2,c_2)$ are \emph{isomorphic} if there exists a group isomorphism from $A_1$ to $A_2$ mapping $(a_1,b_1,c_1)$ to $(a_2,b_2,c_2)$.
\end{definition}

If a regular map $\cM'=(A,a,b,c)$ is obtained from a reflexible oriented rotary map $\cM=(G,R,S)$ by the procedure described above,
then we shall say that $\cM'$ is an {\em orientable regularisation} of $\cM$.
It should be observed at this point that the geometric interpretation of the reflexible oriented rotary map $\cM'$ and its orientable regularisation $\cM$ are the same,
and that the oriented rotary map $\cM'=(G,R,S)$ can be reconstructed from $\cM=(A,a,b,c)$ by letting $R=ab$, $S=bc$, and $G=\langle ab,bc\rangle$.

Geometrically, the group $G$ corresponds to the orientation preserving automorphisms of $\cM$ and has index $2$ in $G$.
\begin{definition}
A regular map $(A,a,b,c)$ is called {\em orientable} if $\langle ab,bc\rangle$ has index $2$ in $A$ and {\em non-orientable} otherwise.
\end{definition}
The above discussion shows that a regular map is orientable if and only if it arises as the orientable regularisation of 
some reflexible oriented rotary map. Since the orientable regularisations $\cM_1'$ and $\cM_2'$ of $\cM_1$ and $\cM_2$ are isomorphic if and only if $\cM_1$ and $\cM_2$ 
are isomorphic (as oriented rotary maps), there is a bijective correspondence between the isomorphism classes of 
reflexible oriented rotary maps and the isomorphism classes of orientable regular maps. In particular, our enumeration immediately yields a census of orientable regular maps with at most $3\, 000$ edges (see Theorem~\ref{thm:mainMaps}).

%As we have seen above, only the orientable regular maps arise as orientable regularisation of oriented rotary map.
Besides orientable regularisation, there is also a different procedure that can be applied to certain oriented rotary maps, which yields all non-orientable regular maps.

Let $\cM = (G,R,S)$ be an oriented rotary map. If $b$ is an involution of $G$ such that $R^b=R^{-1}$ and $S^b=S^{-1}$,
then we say that $b$ is an {\em antipodal reflector of $\cM$}; we shall follow the terminology of \cite{ConWil} 
and call $\cM$ {\em antipodal} in this case.

If $b$ is an antipodal reflector of $\cM$, then one can form a non-orientable regular map $(G,Rb,b,bS)$, 
which we shall call the {\em non-orientable regularisation of} $\cM$ {\em with respect to $b$}. 
Conversely, if $\cM'=(G,a,b,c)$ is a non-orientable regular map, then $\cM=(G,ab,bc)$ is an 
oriented rotary map admitting an antipodal reflector $b$, and $\cM'$ is the non-orientable regularisation of $\cM$
with respect to $b$.

Note that  non-orientable regularisations of $\cM$ that correspond to distinct antipodal reflectors are never isomorphic. Indeed, if $b_1$ and
$b_2$ are two antipodal reflectors of $(G,R,S)$ and if the corresponding non-orientable regularisations $(G,Rb_1,b_1,b_1S)$ and
$(G,Rb_2,b_2,b_2S)$ are isomorphic via an automorphism $\varphi$ of $G$, then $b_2=\varphi(b_1)$, and thus $\varphi(R) = \varphi(R){b_2}^2= \varphi(Rb_1)b_2 = R{b_2}^2=R$; similarly $\varphi(S) = S$, and since $G=\langle R,S\rangle$, this shows that $\varphi$ is trivial and $b_2=b_1$.
Moreover, two antipodal reflectors always differ by a central involution, implying that the number of non-isomorphic non-orientable regularisations
arising from an antipodal oriented rotary map $(G,R,S)$ is one more than the number of involutions in the centre of $G$. 
This phenomenon was first observed in~\cite{WilNOmaps}.

Let us point out here that a non-orientable regular map $(G,a,b,c)$
also has a geometric interpretation, in which vertices, edges and  faces 
correspond to the cosets of the subgroups $\langle b,c \rangle$, $\langle a,c \rangle$, and $\langle b,c\rangle$ in $G$, respectively,
and with the incidence between these objects given with non-empty intersection.
The underlying surface of the map is in this case non-orientable.

With this geometric interpretation in mind,
 the non-orientable regularisation $\cM'$ of an antipodal oriented rotary map $\cM$ is obtained as
the quotient by a central involution in $\Aut(\cM)$ that acts as an orientation reversing homeomorphism of the underlying surface
(see \cite[Proof of Theorem]{ConWil}), and conversely, $\cM$ is the unique orientable smooth $2$-cover of $\cM'$.

The discussion above suggests an obvious strategy  to construct all non-orientable regular maps: construct all oriented rotary maps then, for each oriented rotary map, find all of its antipodal reflectors and then, for each such reflector, construct the corresponding non-orientable regularisation.

In this correspondence, an antipodal oriented rotary map with $m$ edges yields a non-orientable regular map with $m/2$ edges. Hence our database of
oriented rotary maps with at most $3\, 000$ edges yields a complete list of non-orientable regular maps with at most $1\, 500$ edges.
The following theorem summarises the results of our computations.

\begin{theorem}
\label{thm:mainMaps2}
There are precisely $14\,375$ %(1 on one edge - loop on projective plane) 
 non-orientable regular maps with at most $1\, 500$ edges.
\end{theorem}

The number of regular maps with up to a given number of edges is shown in Figure~\ref{fig:NO}.

\begin{figure}[h]
\begin{center}
\includegraphics[scale=0.45]{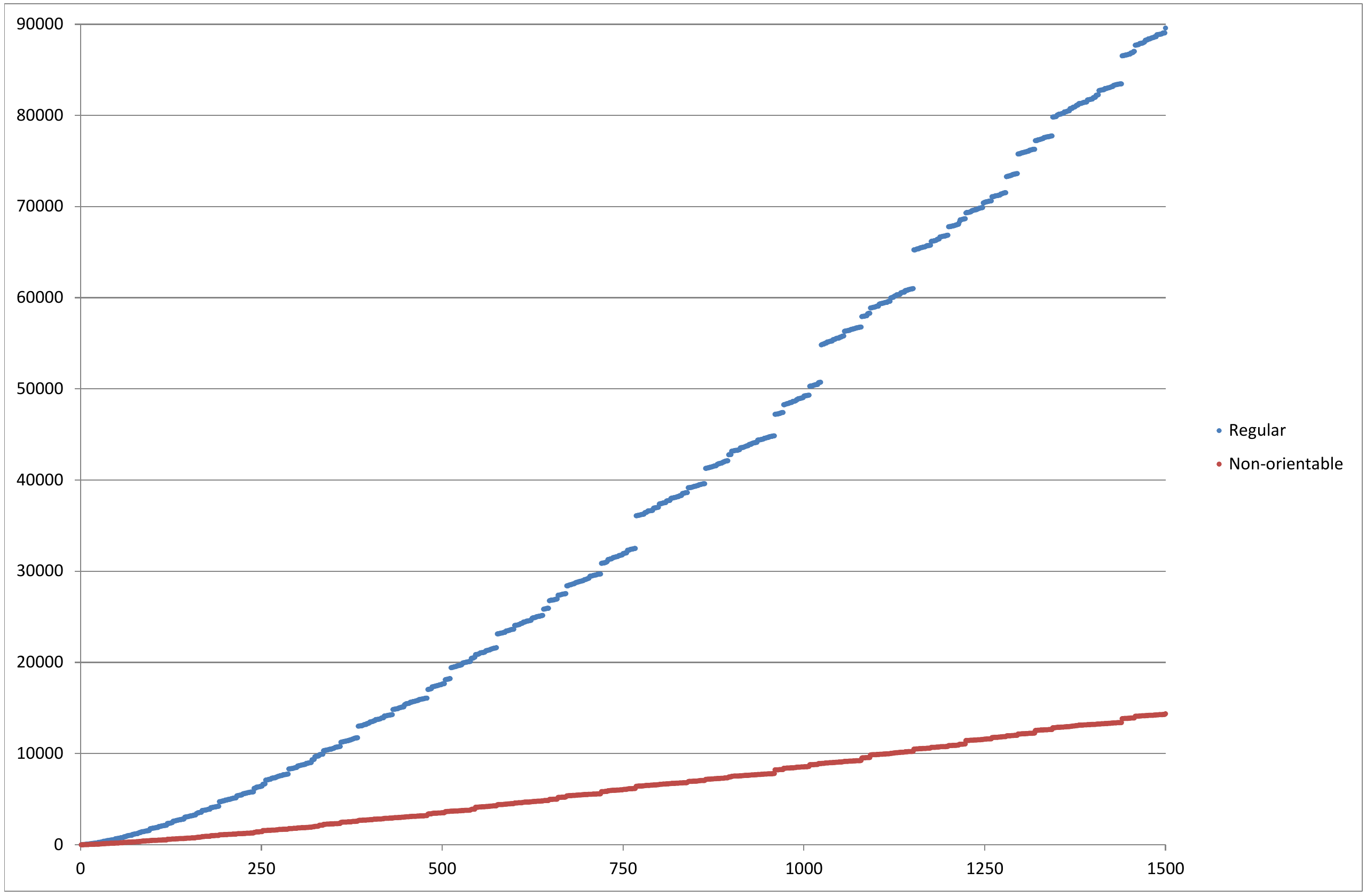}
\caption{\small{Number of all and of non-orientable regular maps with up to a given number of edges.}}
\label{fig:NO}
\end{center}
\end{figure}

%\bigskip
%{\bf Acknowledgement:} We thank Jernej Azarija  and Katja Ber\v{c}i\v{c} for helping us with the computer work related to the
%census of cubic Cayley graph of type I (see Section~\ref{sec:cay}) and of 2-valent digraphs admitting an arc-regular group (see Section~\ref{sec:di}).

%\bibitem{WilOp}
%S.\ E.\ Wilson, Operators over regular maps, 
%{\em Pacific J.\ Math.} {\bf 81} (1979), 559--568.

%\bibitem{PraegerXu} C.~E.~Praeger, M.~Y.~Xu, A Characterization of a Class of Symmetric Graphs of Twice Prime Valency, \textit{European J. Combin.} \textbf{10} (1989), 91--102.

%\bibitem{Tutte} W.~T.~Tutte, A family of cubical graphs, \textit{Proc. Camb. Phil. Soc.} \textbf{43} (1947), 459--474.

%\bibitem{Tutte2} W.~T.~Tutte, On the symmetry of cubic graphs, \textit{Canad. J. Math.} \textbf{11} (1959), 621--624.


\thebibliography{99}

\bibitem{katja} K.~Ber\v{c}i\v{c}, Konstrukcije in katalogizacije simetri\v{c}nih grafov, {\em doctoral dissertation}, University of Ljubljana (2015).

\bibitem{magma} W.~Bosma, J.~Cannon and C.~Playoust, The \texttt{Magma} algebra system. I: The user language, \textit{J. Symbolic Comput.} \textbf{24} (1997), 235--265. 

%\bibitem{ConDob} M.~Conder, P.~Dobcs\'{a}nyi, Trivalent symmetric graphs on up to 768 vertices, \textit{J. Combin. Math. Combin. Comput.} \textbf{40} (2002), 41--63. 

\bibitem{conderPage} M.~Conder, Rotary maps  on closed surfaces with up to $1000$ edges, \href{http://www.math.auckland.ac.nz/~conder/}{www.math.auckland.ac.nz/$\sim$conder/}, accessed April 2015.

\bibitem{ConWil} M.~Conder, S.\ Wilson, Inner reflectors and non-orientable regular maps, {\em Discrete Math.} {\bf 307} (2007), 367--372.

\bibitem{GNU}J.~H.~Conway, H.~Dietrich, E.~A.~O'Brien, Counting Groups: Gnus, Moas, and other Exotica, \textit{Math. Intelligencer}
\textbf{30} (2008), 6--15.
%\bibitem{conway} J.~H.~Conway, A.~Hulpke, J.\ McKay, On transitive permutation groups, \textit{London Math.\  Soc.\ J.\ Comp.\ Math.} {\bf 1} (1998), 1--8.

%\bibitem{ELO} B.~Eick, C.~R.~Leedham-Green and E.~A.~O'Brien, Constructing automorphism groups of $p$-groups, {\em Comm. Algebra} \textbf{30} (2002), 2271--2295. 

%\bibitem{millennium} H.~U.~Besche, B.~Eick, E.~A.~O'Brien, The groups of order at most $2000$, {\em El. Research Announcements of AMS} {\bf 7} (2001), 1--4.

%\bibitem{DrmNed} M.\ Drmota, R.\ Nedela, Asymptotic enumeration of reversible maps regardless of genus, {\em Ars Math.\ Contemporanea} {\bf 5} (2012), 77--97.

\bibitem{Firth} D.~Firth, An algorithm to find normal subgroups of a finitely presented group, up to a given finite index, Ph.D. thesis, University of Warwick, 2005.

\bibitem{GAP}  
The GAP Group, GAP---Groups, Algorithms, and Programming,  Lehrstuhl D
f\"ur Mathematik, RWTH Aachen and School of Mathematical and Computational Sciences,
University of St Andrews (2000), \href{http://www.gap-system.org}{www.gap-system.org}.

%\bibitem{GThandbook} J.\ L.\ Gross, J.\ Yellen, Handbook of Graph Theory, CRC Press (2004).

\bibitem{handbook} D.\ F.\ Holt, B.\ Eick, E.\ O'Brien, Handbook of computational Group Theory, {\em Discrete Mathematics and its applications}, CRC Press (2005).

%\bibitem{isaacs} Isaacs, Finite group theory.

\bibitem{JonSin} G.\ A. Jones, D.\ Singerman, Theory of maps on orientable surfaces, {\em Proc.\ London Math.\ Soc.} {\bf 37} (1978), 273--307.

\bibitem{JonSin2} G.\ A. Jones, D.\ Singerman, Maps, hypermaps and triangle groups. The Grothendieck theory of dessins d'enfants (Luminy, 1993),
 {\em London Math.\ Soc.\ Lecture Note Ser.} {\bf 200}, Cambridge Univ. Press, Cambridge (1994), 115--145.
 
\bibitem{Lub}A.~Lubotzky, Enumerating Boundedly Generated Finite Groups, \textit{J. Algebra} {\bf 238} (2001), 194--199.

\bibitem{LS}A.~Lubotzky, D.~Segal, \textit{Subgroup growth}, Progress in Mathematics 212, Birkh\"{a}user Verlag, 2003.

%\bibitem{MMP} A.~Malni\v{c}, D.~Maru\v{s}i\v{c}, P.~Poto\v{c}nik, Elementary Abelian Covers of Graphs, \textit{J.\ Alg.\ Combin.} {\bf 20} (2004), 71--97.

\bibitem{MS}T.~W.~M\"{u}ller, J.~-C.~Schlage-Puchta, Normal growth of large groups, II, \textit{Arch. Math.} \textbf{84}
(2005), 289--291.

%\bibitem{PSVRestrictive} P.~Poto\v{c}nik, P.~Spiga, G.~Verret, On graph-restrictive permutation groups, \textit{J. Comb. Theory, Ser. B} \textbf{102} (2012), 820--831.


\bibitem{cubiccensus} P.~Poto\v{c}nik, P.~Spiga, G.~Verret,  Cubic vertex-transitive graphs on up to $1280$ vertices, {\em  J.~Symbolic Comput.} {\bf 50} (2013), 465--477.

\bibitem{4HATcensus} P.~Poto\v{c}nik, P.~Spiga, G.~Verret,  A census of $4$-valent half-arc-transitive graphs and arc-transitive digraphs of valence two, {\em  Ars Math. Contemporanea} {\bf 8} (2015), 133--148. 

%\bibitem{bounding} P.~Poto\v{c}nik, P.~Spiga, G.~Verret, Bounding the order of the vertex-stabiliser in $3$-valent vertex-transitive and $4$-valent arc-transitive graphs,    {\em J.\ Combin.\ Theory, Ser.\ B.} {\bf 111} (2015), 148--180.

\bibitem{enumeration} P.~Poto\v{c}nik, P.~Spiga, G.~Verret,  Asymptotic enumeration of vertex-transitive graphs of fixed valency, \url{http://arxiv.org/abs/1210.5736}{arXiv:1210.5736 [math.CO]}

\bibitem{PotWeb} P.~Poto\v{c}nik, P.~Spiga, G.~Verret, {\em Primo\v{z} Poto\v{c}nik's home page}, \url{http://www.fmf.uni-lj.si/~potocnik/work.htm}{http://www.fmf.uni-lj.si/$\sim$potocnik/work.htm}, accessed April 2015.

\bibitem{Serre}J.~-P.~Serre, Trees, Springer-Verlag Berlin Heidelberg 1980.

\bibitem{siran} J.~\v{S}ir\'a\v{n}, Regular Maps on a Given Surface: A Survey,
{\em Topics in Discrete Mathematics Algorithms and Combinatorics} {\bf 26} (2006), 591--609 . 

\bibitem{sage} W.~A.~Stein et al., Sage Mathematics Software (Version 6.4.1), {\em The Sage Development Team} (2015) \url{http://www.sagemath.org}

\bibitem{Tits}J.~Tits, Sur le groupe des automorphismes d'un arbre, \textit{Essays on topology and related topics}, Springer New York, 1970, 188--211.

% \bibitem{WilCan}
%S.\ E.\ Wilson,  Cantankerous maps and rotary embeddings of $K_n$, 
%{\em J. Combin. Theory, Ser. B} {\bf 47} (1989), 262--273.

\bibitem{WilPhD}
S.\ E.\ Wilson,  {\em New Techniques For the Construction of Regular Maps}, PhD Dissertation, University of Washington (1976).


\bibitem{WilNOmaps} 
S.\ E.\ Wilson, Non-orientable regular maps, {\em Ars Combin.} {\bf 5} (1978), 213--218.
\end{document}